\documentclass[legno,11pt, a4paper]{amsart}   	%\usepackage{array}
\usepackage{amssymb}
\usepackage{amsmath}
\usepackage{amsthm}
\usepackage{graphicx}
\usepackage{enumerate}
\usepackage[all]{xy}
\usepackage[parfill]{parskip} 
\usepackage[utf8]{inputenc}
\usepackage[english]{babel}
\usepackage{algorithm}
\usepackage{MnSymbol}
\usepackage{pgf}
\usepackage{tikz}
\usepackage{tikz-cd}
\usetikzlibrary{arrows,automata}
\usetikzlibrary{positioning}
\usepackage[noend]{algpseudocode}
\usepackage{caption}
\usepackage[colorlinks=true,linkcolor=blue]{hyperref}
\usepackage{subcaption}
\usepackage{fancyhdr}
\usepackage{xfrac}
\usepackage[titletoc]{appendix}
\usepackage{booktabs}
\usepackage{url}
\usepackage{breakurl}

\newcommand{\Q}{\mathbb{Q}}
\newcommand{\R}{\mathbb{R}}
\newcommand{\Z}{\mathbb{Z}}
\newcommand{\C}{\mathbb{C}}
\newcommand{\modb}[1]{\,\left(\mathrm{mod}\,{#1}\right)}
\newcommand{\sgn}{\operatorname{sign}}
\newcommand{\conv}{\operatorname{conv}\!}

\theoremstyle{definition}
\newtheorem{thm}{Theorem}[section]
\theoremstyle{definition}
\newtheorem{cor}[thm]{Corollary}
\theoremstyle{definition}

\theoremstyle{definition}
\newtheorem{dfn}[thm]{Definition}
\theoremstyle{definition}
\newtheorem{lem}[thm]{Lemma}
\theoremstyle{definition}
\newtheorem{ex}[thm]{Example}
\theoremstyle{definition}
\newtheorem{conj}[thm]{Conjecture}

\graphicspath{ {images/} }

\begin{document}

\title{Restrictions on the Singularity Content of a Fano Polygon}
\author{Daniel Cavey}

\maketitle

\begin{abstract}
We determine restrictions on the singularity content of a Fano polygon, or equivalently of certain orbifold del~Pezzo surfaces. We establish bounds on the maximum number of $\frac{1}{R}(1,1)$ singularities in the basket of residual singularities. In particular, there are no Fano polygons without T-singularities and with a basket given by (i)~$\big\{ k \times \frac{1}{R}(1,1) \big\}$ for $k \in \Z_{>0}$ and $R \geq 5$, or (ii)~$\big\{ \frac{1}{R_{1}}(1,1), \frac{1}{R_{2}}(1,1), \frac{1}{R_{3}}(1,1) \big\}$.
\end{abstract}

\section{Introduction}

The motivation for this work comes from an approach to classifying del~Pezzo surfaces via Mirror Symmetry that has been introduced in recent years by Coates--Corti--Kasprzyk et al.~\cite{MirrorSymmetryandtheClassificationofOrbifoldDelPezzoSurfaces, MirrorSymmetryandFanoManifolds}. Mirror Symmetry establishes a conjectural relation between certain Laurent polynomials $f \in \C[x_{1}^{\pm 1}, \cdots, x_{n}^{\pm 1}]$ and $n$-dimensional Fano varieties $X$. If $f$ is associated to $X$ under this correspondence then we say that $f$ is \emph{mirror dual} to $X$.

\begin{ex}
The Laurent polynomial $x+y+1/xy$ is known to be mirror dual to $\mathbb{P}^{2}$. The corresponding Newton polytope is:
\begin{center}
$P:=\mathrm{Newt} \left( x+y + \frac{1}{xy} \right) =$
\begin{tikzpicture}[baseline=-0.65ex,scale=0.5, transform shape]
\begin{scope}
\clip (-1.3,-1.3) rectangle (1.3cm,1.3cm); % Clips the picture...
\filldraw[fill=cyan, draw=blue] (0,1) -- (1,0) -- (-1,-1) -- (0,1); % Puts the shaded rectangle
\foreach \x in {-7,-6,...,7}{                           % Two indices running over each
    \foreach \y in {-7,-6,...,7}{                       % node on the grid we have drawn 
    \node[draw,shape = circle,inner sep=1pt,fill] at (\x,\y) {}; % Places a dot at those points
    }
}
 \node[draw,shape = circle,inner sep=4pt] at (0,0) {}; % Places a dot at those points
\end{scope}
\end{tikzpicture}
\end{center}
From the polygon $P$ we construct a toric variety $X_P$ by taking the spanning fan. In this case $X_P$ is again $\mathbb{P}^{2}$.
\end{ex}

In general the toric variety $X_{P}$ associated to the Newton polytope $P$ of $f$ mirror dual to $X$ will be Fano, and it is conjectured that $X_{P}$ admits a \emph{$\Q$-Gorenstein (qG-)} deformation, see \cite{ThreefoldsandDeformationsofSurfaceSingularities, FlipsFlopsandMinimalModels}, to $X$. The toric variety $X_{P}$ may be more singular than $X$, but this is compensated for by being able to use the language of toric geometry to describe the variety. There is an additional complication that the choice of mirror dual is not unique. To this end, Akhtar--Coates--Galkin--Kasprzyk~\cite{MinkowskiPolynomialsandMutations} introduced the notion of \emph{mutation} of a Laurent polynomial: a birational transformation transforming one mirror dual to $X$ to another mirror dual to $X$ {\cite[Lemma 1]{MinkowskiPolynomialsandMutations}. This notion of mutation for Laurent polynomials can be translated to a combinatorial operation on lattice polytopes. See~\cite{MinkowskiPolynomialsandMutations, MinimalityandMutationEquivalenceofPolygons} for the details. An important open question is to begin to classify mutation-equivalence classes of polytopes.

From the viewpoint of Mirror Symmetry, it is natural to restrict ourselves to the study of \emph{Fano} polytopes. Recall that a full-dimensional lattice polytope $P$ is Fano if the vertices $\mathcal{V}(P)$ are all primitive and if the origin lies in the strict interior of $P$. The Newton polytope of any Laurent polynomial mirror is necessarily Fano. Furthermore, when considering the spanning fan, it makes sense to restrict oneself to Fano polytopes. For an overview of Fano polytopes see~\cite{FanoPolytopes}.

An important mutation invariant of a Fano polygon is its singularity content, introduced by Akhtar--Kasprzyk~\cite{SingularityContent}. In order to describe this invariant, we first recall the definition of a cyclic quotient singularity. Consider the action of the cyclic group of order $R$, denoted $\mu_{R}$, on $\C^{2}$ via $(x,y) \mapsto (\epsilon^{a}x,\epsilon^{b}y)$. Here $\epsilon$ is a primitive $R$-th root of unity.  A \emph{quotient singularity} $\frac{1}{R}(a,b)$ is defined by the germ of the origin of $\mathrm{Spec}(\C[x,y]^{\mu_{R}})$. See~\cite{YoungPersonsguidetoCanonicalSingularities} for further details.

\begin{ex}
Consider a $\frac{1}{2}(1,1)$ singularity. Let $G= \Z / 2\Z$ and $\epsilon = -1$. The action of $G$ on $\C^{2}$ is given by $-1 \cdot (x,y) = (-x,-y)$, and
\begin{align*}
\mathrm{Spec}\left(\C[x,y]^{G} \right) &= \mathrm{Spec} \left( \C[x^{2},xy,y^{2}] \right) \\
& = \mathrm{Spec} \left( \C[u,v,w] / (uw-v^{2}) \right) \\
 & = \mathbb{V}(uw-v^{2}) \subset \C^{3}.
\end{align*}
\end{ex}

A quotient singularity $\frac{1}{R}(a,b)$ is cyclic if $\gcd(R,a)=\gcd(R,b)=1$.  Set $k = \gcd(a+b,R)$, so $R=kr$ and $a+b=k \tilde{c}$ for some $r,\tilde{c}\in\Z_{>0}$. The cyclic quotient singularity can be written as $\frac{1}{kr}(1,kc-1)$, where $ca \cong \tilde{c} \modb{R}$ 

Two important classes of cyclic quotient singularities are described by Kollar--Shepherd-Barron~\cite{ThreefoldsandDeformationsofSurfaceSingularities} and Akhtar--Kasprzyk~\cite{SingularityContent}. A cyclic quotient singularity $\frac{1}{kr}(1,kc-1)$ is
\begin{enumerate}
\item[(i)] a \emph{T-singularity} if $r \mid k$;
\item[(ii)] an \emph{R-singularity} if $k<r$. 
\end{enumerate}
In addition, a T-singularity is \emph{primitive} if $r=k$. The significance of these definitions comes when one attempts to smooth the cyclic quotient singularities via a qG-deformation. A cyclic quotient singularity is qG-smoothable if and only if it is a T-singularity, whereas R-singularities are rigid under qG-deformation.

\begin{ex}
A cyclic quotient singularity $\frac{1}{R}(1,1)$ is a T-singularity if and only if $R \in \{1,2,4\}$.
\end{ex}

Consider an arbitrary cyclic quotient singularity $\sigma = \frac{1}{kr}(1,kc-1)$ not necessarily satisfying either $r \mid k$ or $k<r$. There exists unique non-negative integers $n$ and $k_{0}$ such that $k=nr+k_{0}$. If $k_{0} >0$ then $\sigma$ qG-deforms to a $\frac{1}{k_{0}r}(1,k_{0}c-1)$ cyclic quotient singularity. Informally $\sigma$ decomposes as $n$ primitive T-singularities and an R-singularity; the T-part can be smoothed away leaving the R-singularity, which we call the \emph{residue}. More precisely, the residue of $\sigma$ is given by: 
\[ \mathrm{res}(\sigma) = \begin{cases} \varnothing,&\text{ if } k_{0} = 0;\\ \frac{1}{k_{0}r}(1,k_{0}c-1),&\text{ otherwise.} \end{cases}  \]
With notation as above, the \emph{singularity content} of $\sigma$ is denoted by the pair: 
\[ \mathrm{SC}(\sigma) = \left(n, \mathrm{res}(\sigma)\right). \]

Associated to a cyclic quotient singularity $\sigma = \frac{1}{R}(a,b)$ is a cone $C_{\sigma} = \mathrm{cone} \left( e_{1}, e_{2} \right)$ in the lattice $\Z^{2} + (a/r,b/r) \cdot \Z$, defined up to a change of basis. By abuse of notation we often confuse the distinction between cones and singularities; namely we refer to T-cones, primitive T-cones, R-cones and $\frac{1}{R}(a,b)$-cones. 

Given a cone $C_{\sigma} \subset N_{\R} = N \otimes \R$ where $N \cong \Z^{2}$, let $\rho_{1}, \rho_{2}$ be the primitive lattice points generating the rays of $C_{\sigma}$. There is a unique hyperplane $H$ through $\rho_{1}, \rho_{2}$, and $E = C_{\sigma} \cap H$ is the edge over which $C$ is defined. The decomposition of $\sigma$ has a description in the combinatorics of $E$. 

The \emph{lattice length}, denoted $l(E)$, of $E \subset N_{\R}$ is given by the value $\lvert E \cap N \rvert - 1$. The \emph{lattice height} $h(E)$ of $E$ is given by the lattice distance from the origin: that is, given the unique primitive inward pointing normal $n_{E} \in M= \mathrm{Hom}(N,\Z)$ of $E$, the height is given by $\lvert \langle v,n_{E} \rangle \rvert$, for any $v \in E$. There exist $n,r \in \Z_{\geq 0}$ such that $l=hn+r$. Divide $C$ into separate sub-cones $C_{0},\ldots,C_{n}$, where $C_{1},\ldots,C_{n}$ have lattice length $h$, and $C_{0}$ has lattice length $r$. The cones $C_{i}$, for $1 \leq i \leq n$, are primitive T-cones and $C_{0}$ is an R-cone.
\begin{figure}[H]
\begin{center}
\begin{tikzpicture}[transform shape]
\begin{scope}
\clip (-5.7,-0.3) rectangle (2.3cm,3.7cm); % Clips the picture...
\filldraw[fill=cyan, draw=blue] (0,0) -- (-5,3)--(-2,3)--(0,0); % Puts the shaded rectangle
\filldraw[fill=cyan, draw=red] (0,0) -- (-2,3) -- (1,3) -- (0,0);
\filldraw[fill=cyan, draw=blue] (0,0) -- (1,3) -- (2,3) -- (0,0);
\filldraw[fill=cyan, draw=red] (0,0) -- (1,3);
\filldraw[fill=cyan, draw=blue] (0,0) -- (2,3);
\filldraw[fill=cyan, draw=blue] (-5,3) -- (2,3);
\draw[<->, color = cyan](-5.2,0) -- node[cyan, scale = 1.2, anchor = east]{h}(-5.2,3);
\draw[<->, color = cyan](-5,3.2) -- node[cyan, scale = 1.2, anchor = south]{l}(2,3.2);
\foreach \x in {-7,-6,...,7}{                           % Two indices running over each
    \foreach \y in {-7,-6,...,7}{                       % node on the grid we have drawn 
    \node[draw,shape = circle,inner sep=1pt,fill] at (\x,\y) {}; % Places a dot at those points
    }
}
 \node[draw,shape = circle,inner sep=4pt] at (0,0) {}; % Places a dot at those points
\end{scope}
\end{tikzpicture}
\caption{Division of a cone of lattice length 7 and lattice height 3.}
\end{center}
\end{figure}
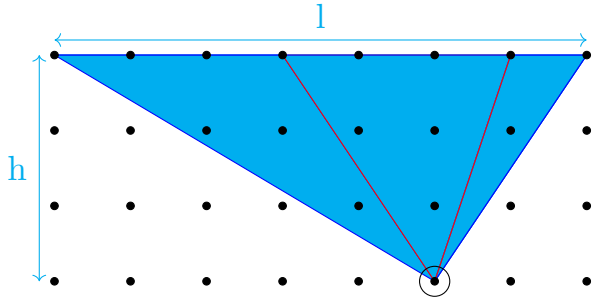

\begin{dfn}[{\cite[Definition 3.1]{SingularityContent}}]
Let $P \subset N_{\R}$ be a Fano polygon. Label the edges of $P$ clockwise $E_{1},\ldots ,E_{k}$. Let $C_{\sigma_{i}}$ be the cone over the edge $E_{i}$. Set 
\[ \mathrm{SC} \left( \sigma_{i} \right) = \left( n_{i},\mathrm{res} \left( \sigma_{i} \right) \right). \] 
Define the \emph{singularity content} of $P$ to be:
\[ \mathrm{SC}(P) = \left( \sum\limits_{i=1}^{k} n_{i},\mathcal{B} \right), \]
where $\mathcal{B} = \left\{ \mathrm{res} \left( \sigma_{1} \right) , \ldots, \mathrm{res} \left( \sigma_{k} \right) \right\}$ is a cyclically ordered set known as the \emph{basket of residual singularities}.
\end{dfn}

The singularity content of $P$, a combinatorial property, describes the singularities on $X_{P}$, a geometrical property.

\begin{dfn}[{\cite[Definition 1]{MirrorSymmetryandtheClassificationofOrbifoldDelPezzoSurfaces}}]
A del~Pezzo surface with cyclic quotient singularities is of class TG if it admits a qG-degeneration with reduced fibres to a normal toric del~Pezzo surface.
\end{dfn}

Del Pezzo surfaces of class TG are exactly those which (conjecturally) can be described by this application of Mirror Symmetry:

\begin{conj}[{\cite[Conjecture A]{MirrorSymmetryandtheClassificationofOrbifoldDelPezzoSurfaces}}]
There exists a bijective correspondence between the set of mutation-equivalence classes of Fano polygons and the set of qG-deformation equivalence classes of locally qG-rigid TG del~Pezzo surfaces with cyclic quotient singularities.
\end{conj}

Recent results support this conjecture~\cite{DelPezzoSurfaceswithaSingle1k11singularity,DelPezzoSurfaceswith1/311points,MinimalityandMutationEquivalenceofPolygons}, and understanding the possible values taken by the singularity content is an important open question. As a first step towards addressing this question, the two main results of this paper are:

\begin{thm}\label{1.8}
There are no Fano polygons with singularity content
\[
\left( 0, \left\{ k \times \frac{1}{R}(1,1) \right\} \right), \qquad\text{ where }k\in\Z_{>0},R \in \Z_{\geq 5}.
\] 
\end{thm}

\begin{thm}\label{1.9}
There are no Fano polygons with singularity content 
\[ \left( 0 , \left\{ \frac{1}{R_{1}} \left( 1,1 \right), \frac{1}{R_{2}} \left( 1,1 \right), \frac{1}{R_{3}}\left( 1,1\right) \right\} \right), \qquad \text{ where } R_{i} \in \{3\} \cup \Z_{\geq 5}. \]
\end{thm}

\noindent These two theorems are conjecturally equivalent to:
\begin{itemize}
\item There are no del~Pezzo surfaces admitting a toric degeneration whose topological Euler number is $0$ and singular locus consists of only isolated $\frac{1}{R}(1,1)$ cyclic quotient singularities, where $R \in \Z_{\geq 5}$;
\item There are no del~Pezzo surfaces admitting a toric degeneration whose topological Euler number is $0$ and singular locus consists of exactly three qG-rigid isolated cyclic quotient singularities $\frac{1}{R_{1}}(1,1)$, $\frac{1}{R_{2}}(1,1)$ and $\frac{1}{R_{3}}(1,1)$.
\end{itemize}

\section{Restrictions via Matrices}\label{Section2}

Let $P \subset N_{\R}$ be a Fano polygon with vertices $v_{1}, v_{2},\ldots,v_{k}$ labelled anticlockwise. By convention our subscripts are considered modulo $k$ to be in the range $\left\{ 1,\ldots, k \right\}$. Consider the set of matrices $\left\{ M_{i} \in GL_{2} \left( \Z \right) \right\}_{1 \leq i \leq k}$ satisfying
\[
M_{i} v_{i} = v_{i+1}, \qquad M_{i} v_{i+1} = v_{i+2}.
\]
It follows that 
\[  M_{k}M_{k-1} \cdots M_{1} = Id. \]

By understanding the matrix $M_{i}$ when $\mathrm{span}_{\R_{\geq 0}}(v_{i},v_{i+1})$ and $\mathrm{span}_{\R_{\geq 0}}(v_{i+1},v_{i+2})$ describe particular cyclic quotient singularities, we create restrictions on when this equality can hold. We start with the simple case of a polygon $P$ consisting entirely of $\frac{1}{3}(1,1)$ cones (we already know exactly one such polygon exists by Kasprzyk--Nill--Prince~\cite{MinimalityandMutationEquivalenceofPolygons}).

Let $E$ be an edge of a Fano polygon such that the cone over $E$ is a $\frac{1}{3}(1,1)$ cone. Assume without loss of generality that E has vertices $\begin{pmatrix} -1 & 3 \end{pmatrix}^{\intercal}$ and $\begin{pmatrix} -2 & 3 \end{pmatrix}^{\intercal}$. Further suppose that the edge adjacent to $E$ sharing the vertex $\begin{pmatrix} -2 & 3 \end{pmatrix}^{\intercal}$ also has a corresponding $\frac{1}{3}(1,1)$ cone.

The corresponding matrix $M=\begin{pmatrix} a & b \\ c & d \end{pmatrix} \in GL_{2}(\Z)$ satisfies:
\[
M \begin{pmatrix} -1 \\ 3 \end{pmatrix} = \begin{pmatrix} -2 \\ 3 \end{pmatrix} \qquad\text{ and }\qquad
\mathrm{det}(M)=1.
\]

The second condition follows from the fact that $M$ maps a $\frac{1}{3}(1,1)$ cone onto a $\frac{1}{3}(1,1)$ cone, and so lattice length and lattice height must be preserved. Hence $M$ is of the form 
\[ M = \begin{pmatrix} 3-2a & \frac{1-2a}{3} \\ 3a-3 & a \end{pmatrix}, \qquad \text{ for } a \in \Z. \]
The only remaining restriction is that $(1-2a)/3 \in \Z$ and so $a \equiv 2 \modb{3}$. Substituting $a=3n+2$, we obtain the set of matrices:
\[ A_{n} = \begin{pmatrix} -6n-1 & -2n-1 \\ 9n+3 & 3n+2 \end{pmatrix}, \qquad \text{ for } n \in \Z. \]
The image of the point $\begin{pmatrix} -2 & 3 \end{pmatrix}^{\intercal}$ under $A_{n}$, that is the second vertex of the second $\frac{1}{3}(1,1)$ cone, is given by:
\begin{align*}
A_{n} \begin{pmatrix} -2 \\ 3 \end{pmatrix} &= \begin{pmatrix} 6n-1 \\ -9n \end{pmatrix}.
\end{align*}
Note if $n<0$, convexity of the Fano polygon is broken. Therefore we have a 1-dimensional family of suitable matrices parametrised by $\Z_{\geq 0}$ each giving a point $v$ such that $\text{span}_{\R_{\geq 0}}\left( \begin{pmatrix} -2 & 3 \end{pmatrix}^{\intercal} ,v \right)$ is a cone representing a $\frac{1}{3}(1,1)$ singularity.

\begin{align*}
n=0 \longleftrightarrow& \begin{pmatrix} -1 \\ 0 \end{pmatrix} \\
n=1 \longleftrightarrow& \begin{pmatrix} 5 \\ -9 \end{pmatrix} \\
n=2 \longleftrightarrow& \begin{pmatrix} 11 \\ -18 \end{pmatrix} \\
\vdots&
\end{align*}

\begin{lem}\label{4.1}
A Fano polygon consisting only of $\frac{1}{3}(1,1)$ R-cones satisfies
\[ M_{k} M_{k-1} \cdots M_{1} = A_{n_{1}} A_{n_{2}} \cdots A_{n_{k}}. \]
\end{lem}

\begin{proof}
We have that
\begin{align*}
M_{1} =& A_{n_{1}} \\
M_{2} =& M_{1} A_{n_{2}} M_{1}^{-1} = A_{n_{1}}A_{n_{2}}A_{n_{1}}^{-1} \\
M_{3} =& M_{2} M_{1} A_{n_{2}} M_{1}^{-1} M_{2}^{-1} = A_{n_{1}} A_{n_{2}} A_{n_{3}} A_{n_{2}}^{-1} A_{n_{1}}^{-1} \\
\vdots& \\
M_{i} =& A_{n_{1}} A_{n_{2}} \cdots A_{n_{i-1}} A_{n_{i}} A_{n_{i-1}}^{-1} \cdots A_{n_{2}}^{-1} A_{n_{1}}^{-1} 
\end{align*}
The identity follows by substitution.
\end{proof}

The problem remains to test when the identity $A_{n_{1}} A_{n_{2}} \cdots A_{n_{k}} =Id$ holds. First consider the $A_{n_{i}}$ modulo $3$:
\begin{align*} 
A_{n_{1}} \equiv& \begin{pmatrix} 2 & 2n_{1}+1 \\ 0 & 2 \end{pmatrix} \modb{3}, \\
A_{n_{1}} A_{n_{2}}  \equiv& \begin{pmatrix} 2 & 2n_{1}+1 \\ 0 & 2 \end{pmatrix} \begin{pmatrix} 2 & 2n_{2}+1 \\ 0 & 2 \end{pmatrix} \equiv \begin{pmatrix} 1 & n_{1}+n_{2}+1 \\ 0 & 1 \end{pmatrix} \modb{3},  \\
A_{n_{1}} A_{n_{2}} A_{n_{3}}  \equiv& \begin{pmatrix} 1 & n_{1}+n_{2}+1 \\ 0 & 1 \end{pmatrix} \begin{pmatrix} 2 & 2n_{3}+1 \\ 0 & 2 \end{pmatrix} \\  
\equiv& \begin{pmatrix} 2 & 2(n_{1}+n_{2}+n_{3}) \\ 0 & 1 \end{pmatrix} \modb{3},  \\
\vdots&
\end{align*}

Note that the multiplication of an odd number of matrices can never equal the identity matrix modulo $3$, since the upper left entry is $2 \not\equiv 1 \modb{3} $. Indeed this follows by noting $A_{n_{7}} A_{n_{6}} \cdots A_{n_{1}} =A_{n_{1}+\cdots+ n_{7}}$ and then induction. Therefore if $A_{n_{1}} A_{n_{2}} \cdots A_{n_{k}} =Id$, it follows that $k$ is even. Looking modulo $9$ further narrows down the possibilities. We have that:
\begin{align*} 
A_{n_{1}} A_{n_{2}} \equiv& \begin{pmatrix} * & * \\ 6 & * \end{pmatrix} \not\equiv Id \modb{9}, \\
A_{n_{1}} A_{n_{2}}A_{n_{3}}A_{n_{4}} \equiv& \begin{pmatrix} * & * \\ 3 & * \end{pmatrix} \not\equiv Id \modb{9}.
\end{align*}
Therefore the smallest possible value of $k$ satisfying $A_{n_{1}} A_{n_{2}} \cdots A_{n_{k}} =Id$ is $6$.

Finally we use the fact that for a Fano polygon, the boundary is a closed loop that wraps around the origin once. We shall use the winding number defined in~\cite{LatticePolygonsandNumber12}, which we now describe.

Considering $SL_{2}(\R)$ as a topological space, the fundamental group is given by $\pi_{1}\left( SL_{2}(\R) \right) = \Z$. The universal cover, denoted $ \widetilde{SL_{2}(\R)}$, is the connected topological group fitting into the exact sequence:
\[ 0 \longrightarrow \Z \longrightarrow \widetilde{SL_{2}(\R)} \longrightarrow SL_{2}(\R) \longrightarrow 0. \]
There is no description of $ \widetilde{SL_{2}(\R)}$ as a group of matrices subject to some algebraic conditions. The commonly used description is that of pairs $\left( M, [\gamma] \right)$, where
\[
M= \begin{pmatrix} a & b \\ c & d \end{pmatrix} \in SL_{2}(\Z)
\]
and $[\gamma]$ is a homotopy equivalence class of paths in $\R^{2} \backslash \{ \mathbf{0} \}$ from $\begin{pmatrix} 0 & 1 \end{pmatrix}^{\intercal}$ to $\begin{pmatrix} c & d \end{pmatrix}^{\intercal}$. Hence $ \widetilde{SL_{2}(\R)}$ has the structure of a group via the composition law
\[ \left( M_{1}, [\gamma_{1}] \right) \cdot \left( M_{2}, [\gamma_{2}] \right) = \left( M_{1} M_{2} , [\gamma_{2} \star \gamma_{1}] \right),  \]
where $\star$ denotes concatenation.  

Define $\widetilde{SL_{2}(\Z)}$ to be the inverse image of $SL_{2}(\Z)$ under $\widetilde{SL_{2}(\R)} \rightarrow SL_{2}(\R)$. Note this is not a covering space of $SL_{2}(\Z)$ since it is not a connected topological space. Lift each matrix $A_{n_{i}}$ to $\widetilde{SL_{2}(\Z)}$ by equipping with an appropriate straight line path denoted $\gamma_{i}$. The algebraic condition on the $A_{n_{i}}$ then becomes
\begin{equation}\label{eq1}
\left( A_{n_{1}}, [\gamma_{1}] \right) \cdot \left( A_{n_{2}}, [\gamma_{2}] \right) \cdot \cdots \cdot \left( A_{n_{k}}, [\gamma_{k}] \right) = \left( Id, [\text{anticlockwise loop}] \right).
\end{equation}
In Poonen--Rodriguez-Villegas~\cite{LatticePolygonsandNumber12}, a homomorphism $\Phi: \widetilde{SL_{2}(\Z)} \rightarrow \Z$ is introduced to act as a winding number. The aim is to apply $\Phi$ to both sides of the above equality to obtain an extra condition on $k$.

Similarly to how $SL_{2}(\Z)$ is generated by 
\[ S=\begin{pmatrix} 0 & -1 \\ 1 & 0 \end{pmatrix} \qquad \text{ and } \qquad T= \begin{pmatrix} 1 & 1 \\ 0 & 1 \end{pmatrix}, \]
it is known that $\widetilde{SL_{2}(\Z)}$ is generated by the two elements $\tilde{S}$ and $\tilde{T}$ obtained from lifting $S$ and $T$ to $ \widetilde{SL_{2}(\R)}$ by equipping them with the straight line path from $\begin{pmatrix} 0 & 1 \end{pmatrix}^{\intercal}$ to $\begin{pmatrix} 1 & 0 \end{pmatrix}^{\intercal}$ and the trivial path respectively. Furthermore it is shown in~\cite{LatticePolygonsandNumber12} that:
\[ \Phi (\tilde{S})= -3 \qquad \text{ and } \qquad \Phi(\tilde{T})= 1. \]
It is routine to check that $(\tilde{S})^{4} = \left( Id, [\text{anticlockwise loop}] \right)$ and so 
\[ \Phi \left( Id, [\text{anticlockwise loop}] \right) = -12. \]
It remains to calculate $\Phi \left( A_{n_{i}}, [\gamma_{i}] \right)$. By using an algorithm of Conrad~\cite{SL2Z}, we obtain the expression:
\[ A_{n_{i}} = T S^{-1}  T^{-2} S^{-1} T^{-(n_{i}+1)} S T^{-3}. \]
After lifting to $\widetilde{SL_{2}(\Z)}$ and applying the winding number homomorphism we obtain: 
\[ \Phi \left( A_{n_{i}}, [\gamma_{i}] \right) = -2-n_{i}. \] 
Applying $\Phi$ to both sides of~\eqref{eq1} gives the expression:
\[ \sum\limits_{i=1}^{k} n_{i} =12-2k. \]
If $k>6$ this implies $ \sum\limits_{i=1}^{k} n_{i} < 0$, but convexity demands the $n_{i}$ to be positive and so there are no solutions. The only remaining case is $k=6$, for which the equation becomes 
\[ \sum\limits_{i=1}^{k} n_{i} =0. \]
Therefore there is a single possible solution given by $k=6$ and $n_{i}=0$. This recovers the known Fano polygon consisting of only $6 \times \frac{1}{3}(1,1)$ R-singularities.

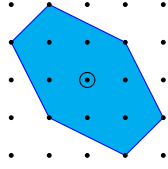
\begin{figure}[H]
\begin{center}
\begin{tikzpicture}[scale=0.5, transform shape]
\begin{scope}
\clip (-2.3,-2.3) rectangle (2.3cm,2.3cm); % Clips the picture...
\filldraw[fill=cyan, draw=blue] (-1,2) -- (1,1) -- (2,-1) -- (1,-2) -- (-1,-1) -- (-2,1) -- (-1,2); % Puts the shaded rectangle
\foreach \x in {-7,-6,...,7}{                           % Two indices running over each
    \foreach \y in {-7,-6,...,7}{                       % node on the grid we have drawn 
    \node[draw,shape = circle,inner sep=1pt,fill] at (\x,\y) {}; % Places a dot at those points
    }
}
 \node[draw,shape = circle,inner sep=4pt] at (0,0) {}; % Places a dot at those points
\end{scope}
\end{tikzpicture}
\end{center}
\caption{Fano polygon with singularity content $\left( 0, \left\{ 6 \times \frac{1}{3}(1,1) \right\} \right)$}.
\end{figure}

\noindent We now generalise our approach to prove Theorem~\ref{1.8}.

\begin{proof}[Proof of Theorem~\ref{1.8}]
First assume $R$ odd. Consider the standard position of the $\frac{1}{R}(1,1)$ cone to have vertices $\left(-(R+1)/2,R \right)$ and $ \left(-(R-1)/2,R \right)$. Then we are looking for $A \in SL_{2}(Z)$ such that
\[ A \begin{pmatrix} -\frac{R-1}{2} \\ R \end{pmatrix} = \begin{pmatrix} -\frac{R+1}{2} \\ R \end{pmatrix} \qquad \text{ and } \qquad \mathrm{det}(A)=1.\]
Such a matrix takes the form
\[ A  =\begin{pmatrix} \frac{-aR-a+2R}{R-1} & \frac{R-aR-a-1}{2R} \\ \frac{2R(a-1)}{R-1} & a \end{pmatrix}. \]
The entries of $A$ belong to $\Z$ if and only if 
\[ a \equiv 1 \modb{(R-1)/2}, \qquad \text{ and } \qquad a \equiv -1 \modb{R} .\]
This implies that $a=2R-1+n ((R-1)R)/2$ for some $n \in \Z$. Making this substitution into $A$ gives:
\[ A_{n}^{(R)} = \begin{pmatrix} -n\frac{R+1}{2}R-2R-1 & -n\frac{R^{2}-1}{4} - R \\ nR^{2}+4R & 2R-1+n\frac{R-1}{2}R \end{pmatrix}. \]
The problem is reduced to testing when the identity  $A_{n_{1}}^{(R)} A_{n_{2}}^{(R)} \cdots A_{n_{k}}^{(R)} =Id$ can hold using a generalised version of Lemma~\ref{4.1}. Studying $A_{n_{i}}^{(R)}$ modulo $R$:
\begin{align*}
A_{n_{1}}^{(R)} \equiv& \begin{pmatrix} -1 & (\frac{R^{2}-1}{4})n_{1} \\ 0 & -1 \end{pmatrix} \not\equiv Id \modb{R}, \\
A_{n_{1}}^{(R)} A_{n_{2}}^{(R)}  \equiv& \begin{pmatrix} 1 & \frac{R^{2}-1}{4}(-n_{1}-n_{2}) \\ 0 & 1 \end{pmatrix} \equiv Id \text{ , if } n_{1}+n_{2}\equiv 0 \modb{R},  \\
A_{n_{1}}^{(R)} A_{n_{2}}^{(R)} A_{n_{3}}^{(R)}  \equiv& \begin{pmatrix} -1 & \frac{R^{2}-1}{4}(n_{1}+n_{2}+n_{3}) \\ 0 & -1 \end{pmatrix} \equiv A_{n_{1}+n_{2}+n_{3}}^{(R)} \not\equiv Id \modb{R}.  \\
\end{align*} 
Continuing inductively there cannot be a solution if $k$ is odd. Furthermore for $k$ even, the identity $A_{n_{1}}^{(R)} A_{n_{2}}^{(R)} \cdots A_{n_{k}}^{(R)} =Id$ holds only if $\sum\limits_{i=1}^{k} n_{i} \equiv 0 \modb{R}$. Studying the product of $A_{n_{i}}$'s modulo $R^{2}$, observe that:
\[ \prod_{i=1}^{k} A_{n_{i}}^{(R)} \equiv \begin{pmatrix} * &\ *\  \\ (-1)^{k}4kR &\ *\ \end{pmatrix} \modb{R^2}, \]
and so $A_{n_{1}}^{(R)} A_{n_{2}}^{(R)} \cdots A_{n_{k}}^{(R)} =Id$ holds only if $k$ is a multiple of $R$. The smallest possible value for $k$ is $2R$. Finally, appealing to the winding number argument calculate that:
\[ A_{n}^{(R)} = T S^{-1} (T^{-2}S^{-1})^{\frac{R-3}{2}} T^{-2} S^{-1} T^{-(n+2)} S^{-1} (T^{-2}S^{-1})^{\frac{R-5}{2}} T^{-2} S T^{-3},  \]
and so $\Phi \left( A_{n}^{(R)}, [\gamma_{i}] \right) =-6+R-n$. Applying $\Phi$ to~\eqref{eq1} obtain 
\[ \sum\limits_{i=1}^{k} n_{i} =12 - \left( R-6 \right)k. \]
Since $k$ must be a multiple of $2R$, and $\sum\limits_{i=1}^{k} n_{i}$ must be congruent to $0$ modulo $R$,
\[ 12 \equiv 0 \modb{R}. \]
This implies $R \mid 12$ and since $R$ is odd and greater or equal 5, there are no solutions. The case $R$ even follows similarly.
\end{proof}

Note that $R=3$ satisfies the congruence $12 \equiv 0 \modb{R}$ corresponding to the fact that there is a solution in this case

\section{Restrictions via Continued Fractions}

In this section, we use results on continued fractions to prove Theorem~\ref{1.9}. The geometry of continued fractions can be studied in Karpenkov~\cite{GeometryofContinuedFractions}.

\subsection{Continued Fractions}

\begin{dfn}
For $a_{0}, a_{1}, \cdots, a_{k} \in \R$, consider the \emph{continued fraction}:
\[ [a_{0}: a_{1} : \cdots : a_{k}] = a_{0} + \cfrac{1}{a_{1}+ \cfrac{1}{a_{2}+ \cfrac{1}{\ldots + \frac{1}{a_{k}}}}}. \]
The numbers $a_{i}$ are called the \emph{elements} of the continued fraction. A continued fraction is \emph{odd/even} if there are an odd/even number of elements.
\end{dfn}

There  uniquely exist polynomials $P_{k}$ and $Q_{k}$ in variables $a_{i}$ satisfying:
\[ [a_{0}: a_{1} : \cdots : a_{k}] = \frac{P_{k}(a_{0},\ldots,a_{k})}{Q_{k}(a_{0},\ldots,a_{k})}, \qquad \text{ and } \qquad P_{k}(0,\ldots,0) + Q_{k}(0,\ldots,0) =1. \]
The first few of these polynomials are:
\begin{align*}
[a_{0}] &= \frac{P_{0}(a_{0})}{Q_{0}(a_{0})} = \frac{a_{0}}{1}, \\
[a_{0}:a_{1}] &= \frac{P_{1}(a_{0},a_{1})}{Q_{1}(a_{0},a_{1})} = \frac{a_{0}a_{1}+1}{a_{1}}, \\
[a_{0}:a_{1}:a_{2}] &= \frac{P_{2}(a_{0},a_{1},a_{2})}{Q_{2}(a_{0},a_{1},a_{2})} = \frac{a_{0}a_{1}a_{2}+a_{0}+a_{2}}{a_{1}a_{2}+1}.
\end{align*}
The polynomials $P_{k}$ and $Q_{k}$ satisfy the recursions:
\[ P_{k} = a_{k}P_{k-1} + P_{k-2}, \qquad \text{ and } \qquad Q_{k} = a_{k}Q_{k-1} + Q_{k-2}. \]

\subsection{Integer Geometry}

\begin{dfn}
Consider an integer triangle $\Delta ABC$, that is a triangle whose vertices are the integer points $A$, $B$ and $C$. The \emph{integer area} of $\Delta ABC$, denoted $l \text{Area}(\Delta ABC)$, is given by the index of the sublattice generated by the line segments $AB$ and $AC$ thought of as vectors in the integer lattice.
\end{dfn}

\begin{dfn}
Consider an integer angle $\angle ABC$, that is an angle between two integer lines based at an integer point. The \emph{integer sine} of $\angle ABC$, denoted $l\mathrm{sin}(\angle ABC)$, is given by 
\[ l \mathrm{sin}(\angle ABC) = \frac{l\mathrm{Area}(\Delta ABC)}{l(AB)l(BC)}. \]
\end{dfn}

\begin{dfn}
A \emph{broken line} is defined by $L = A_{0} A_{1} \cdots A_{n} = \bigcup_{i=0}^{n-1} L_{i} $, where $L_{i}$ is the line segment between the integer points $A_{i}$ and $A_{i+1}$. Let $L$ be an integer broken line that does not contain the origin $\mathbf{0} \in \Z^{2}$. If all the $L_{i}$ are at lattice height $1$, then $L$ is called an \emph{$\mathbf{0}$-broken line}.
\end{dfn}

\begin{dfn}
Let $A_{0} A_{1} \cdots A_{n}$ be an $\mathbf{0}$-broken line. Associate to the broken line its \emph{lattice-signed-length-sine (LSLS) sequence} given by $(a_{0},a_{1},\ldots,a_{2n-2})$, where
\begin{align*}
a_{0} =& \sgn(A_{0} \mathbf{0} A_{1} ) \cdot l(A_{0}A_{1}), \\
a_{1} =& \sgn(A_{0} \mathbf{0} A_{1}) \cdot \sgn(A_{1} \mathbf{0} A_{2}) \cdot \sgn(A_{0}A_{1}A_{2}) \cdot l \mathrm{sin} (\angle A_{0}A_{1}A_{2}), \\
a_{2} =& \sgn(A_{1} \mathbf{0} A_{2} ) \cdot l(A_{1}A_{2}), \\
\vdots& \\
a_{2n-3} =& \sgn(A_{n-2} \mathbf{0} A_{n-1}) \cdot \sgn(A_{n-1} \mathbf{0} A_{n}) \cdot  \\
& \qquad  \sgn(A_{n-2}A_{n-1}A_{n}) \cdot l \mathrm{sin} (\angle A_{n-2}A_{n-1}A_{n}), \\
a_{2n-2} =& \sgn(A_{n-1} \mathbf{0} A_{n} ) \cdot l(A_{n-1}A_{n}), \\
\end{align*}
and $\sgn(ABC) = \begin{cases} 1,&\text{ if } (BA,BC) \text{ is orientated positively;} \\ 0,&\text{ if } A,B,C \text{ are collinear;} \\ -1,&\text{ if } (BA,BC) \text{ is orientated negatively.}\end{cases}$
\end{dfn}

Given an $\mathbf{0}$-broken line the LSLS sequence measures the lattice lengths and lattice sine of the angles as we travel along the broken line, up to some change in sign. Since lattice length and lattice sign are invariant under $GL(\Z^{2})$ transformations, so is the LSLS sequence of a broken line.

\subsection{LSLS sequence of Fano Polygons with R-singularities}

Let $C$ be the cone over the edge of a Fano polygon.

\begin{dfn}
The \emph{sail} $S(C)$ of $C$ is given by $\conv\left( C \backslash \{ \mathbf{0} \} \cap \Z^{2} \right)$.
\end{dfn}

\begin{lem}
The boundary $\delta S(C)$ of the sail of a cone $C$ defines an $\mathbf{0}$-broken line.
\end{lem}

\begin{proof}
We need to show that each component $L_{i}$ of the broken line $\delta S(C) = A_{0}A_{1}\cdots A_{n}$ is at lattice height $1$. Consider the line segment $L_{i}$ with vertices $A_{i}$ and $A_{i+1}$. By the definition of $S(C)$, there are no interior points in $\conv\left( \mathbf{0}, A_{i}, A_{i+1} \right)$ which is equivalent to the Euclidean area of $\conv\left( \mathbf{0}, A_{i}, A_{i+1} \right)$ being $1/2$ which is equivalent to the lattice height of $L_{i}$ equalling $1$.
\end{proof}

Using this lemma, associate to a cone an LSLS sequence.

\begin{ex}
Consider a $\frac{1}{R}(1,1)$ R-singularity. First suppose $R$ is even. Since $\frac{1}{2}(1,1)$ and $\frac{1}{4}(1,1)$ are T-singularities, assume $R\geq 6$. Without loss of generality the $\frac{1}{R}(1,1)$ cone has ray generators $(-1,R/2)$ and $(1,R/2)$. The corresponding broken line then has vertices $(-1,R/2)$, $(0,1)$ and $(1,R/2)$ giving a LSLS sequence of $[1:R-2:1]$. The odd case with $R>3$ is treated similarly by considering the cone $C_{\frac{1}{R}(1,1)}$ with ray generators $(- (R+1)/2,R)$ and $(-(R-1)/2,R)$. The LSLS sequence is again $[1:R-2:1]$.
\end{ex}

Observe that the sum of the elements of the LSLS sequence of a $\frac{1}{R}(1,1)$ singularity is equal to $R$; the Gorenstein index. This is not a property that generalises to arbitrary cyclic quotient singularities. Consider a $\frac{1}{9}(1,5)$ cone with rays generated by $(-1,3)$ and $(2,3)$. The LSLS sequence of the cone is $[1:3:2]$, and the sum of the elements is not equal to $9$.

We use the following corollary from~\cite{GeometryofContinuedFractions} alongside the LSLS sequence of a $\frac{1}{R}(1,1)$ cone to find combinations of R-singularities which cannot be glued together to form a Fano polygon.

\begin{cor}[{\cite[Corollary 11.14]{GeometryofContinuedFractions}}]\label{3.16}
Consider a broken line $A_{0}A_{1} \cdots A_{n}$ with the LSLS sequence $(a_{0},a_{1},\ldots,a_{2n})$. Then the broken line is closed if and only if 
\[ P_{2n+1}(a_{0},a_{1},\ldots,a_{2n}) =0 \qquad \text{ and } \qquad Q_{2n+1}(a_{0},a_{1},\ldots,a_{2n}) =1. \]
\end{cor}

\begin{ex}
Consider the unique Fano polygon $P$ with six $\frac{1}{3}(1,1)$ cones. By glueing the broken line of each cone together along each vertex of $P$, obtain a broken line associated to $P$ and through this an LSLS sequence. It is routine to check that the integer sine for each of the angles at a vertex of $P$ is $-1$, and that the LSLS sequence satisfies
\[ \big[ 1:1:1:-1:1:1:1:-1:1:1:1:-1:\cdots:1:1:1 \big] = \frac{0}{1}. \]
\end{ex}

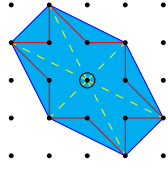
\begin{figure}[H]
\begin{center}
\begin{tikzpicture}[scale=0.5, transform shape]
\begin{scope}
\clip (-2.3,-2.3) rectangle (2.3cm,2.3cm); % Clips the picture...
\filldraw[fill=cyan, draw=blue] (-1,2) -- (1,1) -- (2,-1) -- (1,-2) -- (-1,-1) -- (-2,1) -- (-1,2); % Puts the shaded rectangle
\draw[red] (-1,2) -- (-1,1) -- (-2,1) -- (-1,0) -- (-1,-1) -- (0,-1) -- (1,-2) -- (1,-1) -- (2,-1) -- (1,0) -- (1,1) -- (0,1) -- (-1,2);
\draw[yellow,dashed] (0,0) -- (-1,2);
\draw[yellow,dashed] (0,0) -- (-2,1);
\draw[yellow,dashed] (0,0) -- (-1,-1);
\draw[yellow,dashed] (0,0) -- (1,-2);
\draw[yellow,dashed] (0,0) -- (2,-1);
\draw[yellow,dashed] (0,0) -- (1,1);
\foreach \x in {-7,-6,...,7}{                           % Two indices running over each
    \foreach \y in {-7,-6,...,7}{                       % node on the grid we have drawn 
    \node[draw,shape = circle,inner sep=1pt,fill] at (\x,\y) {}; % Places a dot at those points
    }
}
 \node[draw,shape = circle,inner sep=4pt] at (0,0) {}; % Places a dot at those points
\end{scope}
\end{tikzpicture}
\end{center}
\caption{Broken line associated to the Fano polygon with singularity content $\left( 0, \left\{ 6 \times \frac{1}{3}(1,1) \right\} \right)$.}
\end{figure}

Corollary \ref{3.16} provides a test as to whether it is possible to glue together combinations of R-cones to form a Fano polygon. Namely if there exists a Fano polygon made of cones, cyclically ordered and corresponding to the cyclic quotient singularities $\frac{1}{R_{1}}(1,1), \frac{1}{R_{2}}(1,1), \ldots, \frac{1}{R_{k}}(1,1)$ respectively, then there is a solution to the identity
\[ [1:R_{1}-2:1:m_{1}:1:R_{2}-2:1:m_{2}:\cdots:m_{k-1}:1:R_{k}-2:1]=\frac{0}{1}, \]
where $m_{i} \in \Z$ is the integer sine of the angle of the associated broken line lying between consecutive cones. Furthermore the convexity of the Fano polygon dictates that $m_{i}<0$. Note this variable $m_{i}$ is analogous to the 1-dimensional family of matrices parametrised by $\Z_{\geq 0}$ obtained in Section \ref{Section2}.

The association of a broken line to a polygon is not unique. The choice of starting point for the broken line may change the continued fraction of the broken line since the integer sine of the angle at this point is omitted from the LSLS sequence. However the choice of starting point does not affect that the associated continued fraction should evaluate to $0/1$. Indeed this condition is required to hold at all choices of starting point.

We are now able to prove Theorem~\ref{1.9}.

\begin{proof}[Proof of Theorem~\ref{1.9}]
If such a Fano polygon did exist, then by Corollary~\ref{3.16} there would be a solution $(m_{1},m_{2}) \in \Z^{2}_{<0}$ to the following continued fraction:
\[ [1:R_{1}-2:1;m_{1}:1:R_{2}-2:1:m_{2}:1:R_{3}-2:1] = \frac{0}{1}. \]
By calculating the polynomials $P_{10}(a_{0},\cdots,a_{10})$ and $Q_{10}(a_{0},\cdots,a_{10})$ and substituting appropriately for the $a_{i}$ the condition of the continued fraction translates to the simultaneous equations:
\begin{align*}
P_{10}(1,R_{1}-2,1,m_{1},1,R_{2}-2,&1,m_{2},1,R_{3}-2,1) =\\
&Am_{1}m_{2}+Bm_{1} + Cm_{2} +D =0, \\
Q_{10}(1,R_{1}-2,1,m_{1},1,R_{2}-2,&1,m_{2},1,R_{3}-2,1) =\\
&Em_{1}m_{2}+Fm_{1} + Gm_{2} +H=1,
\end{align*}
where
\begin{align*}
A &= R_{1}R_{2}R_{3}, \\
B &= 2R_{1}R_{2}R_{3} -R_{1}R_{3} - R_{2}R_{3}, \\
C &= 2R_{1}R_{2}R_{3} - R_{1}R_{2} - R_{1}R_{3}, \\
D &= 4R_{1}R_{2}R_{3} -2R_{1}R_{2} -4R_{1}R_{3} -2R_{2}R_{3} +R_{1} +R_{2} +R_{3}-12, \\
E &= R_{1} R_{2} R_{3} -R_{2}R_{3}, \\
F &= 2R_{1} R_{2} R_{3} -R_{1}R_{2} - R_{1}R_{3} -2R_{2}R_{3} + R_{2} + R_{3}, \\
G &= 2R_{1}R_{2}R_{3} - R_{1} R_{3} -3R_{2}R_{3} +R_{3}, \\
H &= 4R_{1}R_{2}R_{3} -2R_{1}R_{2} - 4R_{1}R_{3} - 6R_{2}R_{3} + R_{1} -3R_{2} + 5R_{3} +1.
\end{align*}
Solving the simultaneous equations for $m_{2}$ gives
\[ m_{2} = \frac{-(CF+DE+A-G) \pm \sqrt{(CF+DE+A-G)^{2}-4CE(DF+B-H)}}{2CE}. \]
This expression is not integer for $R_{i} \in \{ 3 \} \cup \Z_{\geq 5}$.
\end{proof}

%Note this is not a property that extends to R-singularities in general. Certainly there exist Fano polygons who singularity content satisfies $n=0$ and $\lvert \mathcal{B} \rvert =3$. For example consider the Fano polygon $P$ with vertices $ \mathcal{V}(P) = \big\{ (-1,2), (1,1), (2,-9) \big\}$. This polygon $P$ corresponds to the weighted projective space $X_{P} = \mathbb{P}(3,5,11)$, which has cyclic quotient singularities $\frac{1}{3}(1,1), \frac{1}{5}(1,2)$ and $\frac{1}{11}(1,9)$. As expected this is reflected in the combinatorics of the Fano polygon: the singularity content of $P$ is given by $\left(0, \{ \frac{1}{3}(1,1), \frac{1}{5}(1,2), \frac{1}{11}(1,9) \} \right)$. Choosing the vertex $(-1,2)$ as the starting point of the broken line, $P$ corresponds to the continued fraction
%\[ [1:1:1:-1:5:1:1:-1:2:1:1] = \frac{0}{1}. \]

\section*{Acknowledgements}

I would like to thank Alexander Kasprzyk, my doctoral advisor, for his guidance. Additionally I am grateful to Takayuki Hibi and Akiyoshi Tsuchiya for their organisation of the Summer Workshop on Lattice Polytopes held in Osaka University 2018. Finally I would like to thank Irem Portakal for helpful discussion at the afore mentioned Workshop. This work was partially supported by Kasprzyk's EPRSC Fellowship EP/NO22513/1.

\bibliographystyle{plain}

School of Mathematical Sciences, University of Nottingham, Nottingham, NG7 2RD, UK. \textit{E-mail address:} danielcavey27@gmail.com

\end{document}